\documentclass{amsart}
\usepackage{amsmath}
\usepackage{amsthm}
\usepackage{amssymb}
\usepackage{color}
\usepackage{float}
\usepackage{graphicx}
\usepackage{tikz}
\usepackage{subfigure}
\usepackage{epstopdf}
\usepackage{cases}
\usepackage{bm}
\usepackage{euscript}
\newcommand{\R}{\mathbb{R}}
\newcommand{\D}{\mathbb{D}}
\newcommand{\C}{\mathbb{C}}
\renewcommand{\H}{\mathcal{H}}

\theoremstyle{plain}
\newtheorem{thm}{Theorem}[section]
\newtheorem*{thm*}{Theorem}
\newtheorem{lem}[thm]{Lemma}
\newtheorem{prop}[thm]{Proposition}
\theoremstyle{definition}

\newtheorem{rmk}[thm]{Remark}
\DeclareMathOperator{\interior}{Int}
\renewcommand{\Re}{\operatorname{Re}}
\renewcommand{\Im}{\operatorname{Im}}
\numberwithin{equation}{section}

\newcommand{\abs}[1]{\lvert#1\rvert}
\def\a{\alpha}              \def\th{\theta}       
      \def\z{\zeta}
\def\zth{\z_{\th}} \def\vth{v_{\th}} 
      \def \ov{\overline}
\def\beq{\begin{equation}}  \def\eeq{\end{equation}}
\def\beqq{\begin{equation*}}  \def\eeqq{\end{equation*}}
\def\bprof{\begin{proof}}    \def\eprof{\end{proof}}
\def\bad{\begin{aligned}}    \def\ead{\end{aligned}}
\def\l{\left}      \def\r{\right}
\def\bthm{\begin{thm}} \def\ethm{\end{thm}}
\def\blem{\begin{lem}} \def\elem{\end{lem}}
\def\p{\partial} \def\e{e^{i \th}}
\def\be{\begin{enumerate}} \def\ee{\end{enumerate}}
  \def\lam{\lambda}

\begin{document}

\title[Variability regions for the second derivative]{Variability regions for the second derivative of bounded analytic functions}

\author{Gangqiang Chen}
\address{Graduate School of Information Sciences,
Tohoku University,
Aoba-ku, Sendai 980-8579, Japan}
\email{cgqmath@ims.is.tohoku.ac.jp}
\author{Hiroshi Yanagihara}
\address{Department of Applied Science \\
        Faculty of Engineering \\
        Yamaguchi University \\
        Tokiwadai, Ube 755-8611 \\
        Japan}
\email{hiroshi@yamaguchi-u.ac.jp}

\begin{abstract}
Let $z_0$ and $w_0$ be given points in the open unit disk $\mathbb{D}$
with $|w_0| < |z_0|$.
Let $\H_0$ be the class of all analytic self-maps $f$
of $\mathbb{D}$ normalized by $f(0)=0$, and
$\H_0 (z_0,w_0) =  \{ f \in \H_0 : f(z_0) =w_0\}$.
In this paper, we explicitly determine the variability region
of $f''(z_0)$ when $f$ ranges over $\H_0 (z_0,w_0)$.
We also show a geometric view of our main result by Mathematica.
\end{abstract}

\maketitle

\section{Introduction}
First we fix some notation. For $c \in \mathbb{C}$ and $r > 0$,
let $\mathbb{D}(c,r ) = \{ z \in \mathbb{C} : |z-c| < r \}$
and
$\overline{\mathbb{D}}(c,r ) = \{ z \in \mathbb{C} : |z-c| \leq r \}$.
In particular we denote the open and closed unit disks
$\mathbb{D}(0,1)$ and $\overline{\mathbb{D}}(0,1)$ by
$\mathbb{D}$ and $\overline{\mathbb{D}}$, respectively. Let $z_0$ and $w_0$ be given points in the open unit disk $\mathbb{D}$
with $|w_0| < |z_0|$. We denote by $\H_0$ the set of all analytic self-maps $f$ of $\mathbb{D}$ normalized by $f(0)=0$ and set $\H_0(z_0,w_0)=\{f\in\H_0: f(z_0)=w_0\}$.
 Schwarz's Lemma states that $\{f(z_0):f \in \mathcal{H}_0\}=\ov{\D}(0,\ |z_0|)$ for any  $z_0\in \mathbb{D}$, and $f(z_0)\in \partial \ov{\D}(0,\ |z_0|)$ if and only if
 $ f(z)=e^{i \theta}z$ for some $\theta \in \mathbb{R}$.

In 1934, Rogosinski \cite{rogosinski1934} explicitly described the region of values of $f(z_0)$
when $f$ ranges over $\H_0$ satisfying
$f'(0) = \mu$ for some prescribed value  $\mu \in \overline{\mathbb{D}}$ (see also \cite{beardon2004multi}, \cite{duren1983univalent}, \cite{mercer1997sharpened}).
This refinement of Schwarz's lemma asserts that for
$z_0 \in \mathbb{D} \backslash \{ 0 \}$,
\[
  \{ f(z_0) : f \in \H_0 \text{ with } f'(0) = \mu \}
  = \overline{\mathbb{D}}\left(\frac{z_0 \mu (1-|z_0|^2)}{1-|z_0|^2|\mu|^2},\frac{(1-|\mu|^2)|z_0|^2}{1-|z_0|^2|\mu|^2}\right).
\]
Notice that the variability region
is strictly contained in $\overline{\mathbb{D}}(0,|z_0|)$.

In 1931, Dieudonn\'{e} \cite{dieudonne1931} determined the variability region
of $f'(z_0)$ at a fixed point $z_0 \in \mathbb{D}\backslash \{0\}$ when $f$ ranges over
$\H_0(z_0,w_0)$. We write
\begin{equation}\label{eq:T-a-z}
T_{a}(z)=\frac{z+a}{1+\overline{a}z},\quad z, a\in\D,
\end{equation}
and define
\[
  \Delta (z_0,w_0)
  =
  \overline{\mathbb{D}}
  \left( \frac{w_0}{z_0} , \frac{|z_0|^2 -|w_0|^2}{|z_0|(1-|z_0|^2)} \right) .
\]
Then he obtained
\begin{equation}\label{1stDieudonne}
   \{ f'(z_0) : f \in \H_0(z_0,w_0) \}
   =
   \Delta (z_0,w_0).
\end{equation}
For $f \in \mathcal{H}_0(z_0,w_0)$ consider the function $\tilde{f}$
defined implicitly by
\begin{equation}\label{eq:tilde-f}
    \frac{z-z_0}{1-\overline{z_0} z}\tilde{f}(z)
   = \frac{\frac{f(z)}{z} - \frac{w_0}{z_0}}
   {1- \overline{\left( \frac{w_0}{z_0}\right)} \frac{f(z)}{z}} .
\end{equation}
Notice that $|\tilde{f}(z)| \leq 1$, $z \in \mathbb{D}$.
Differentiating both sides shows
\begin{equation}
\label{eq:diff-of-g-and-f}
    \frac{1-|z_0|^2}{(1-\overline{z_0} z)^2} \tilde{f}(z)
     + \frac{z-z_0}{1-\overline{z_0} z} \tilde{f}'(z)
   = \frac{1 - \left| \frac{w_0}{z_0}\right|^2 }
   {\left( 1- \overline{\left( \frac{w_0}{z_0}\right)}
     \frac{f(z)}{z}\right)^2}
   \frac{zf'(z)-f(z)}{z^2}.
\end{equation}
By substituting $z =z_0$, we have
\begin{equation}\label{eq:tilde-f-f-z0}
\frac{\tilde{f}(z_0)}{1-|z_0|^2}=\frac{z_0f'(z_0)-w_0}{z_0^2 \l(1-\abs{\frac{w_0}{z_0}}^2\r)},
\end{equation}
and hence
\begin{equation}
\label{eq:diff-of-g-and-f-at-z_0}
   f'(z_0)
   =
   \frac{w_0}{z_0} + \frac{|z_0|^2-|w|^2}{\ov{z}_0(1-|z_0|^2)}
   \tilde{f}(z_0) .
\end{equation}
Combining this and the estimate $|\tilde{f}(z_0)|\leq 1$ we easily obtain
$\{ f'(z_0) : f \in \mathcal{H}_0(z_0,w_0) \} \subset \Delta (z_0,w_0)$.
The reverse inclusion relation follows from considering the function
$f_\lambda \in \mathcal{H}_0(z_0,w_0)$ defined by
$$
f_{\lambda}(z)=z T_{\frac{w_0}{z_0}}\left(\lambda T_{-z_0}(z)\right).
$$
Notice that $f_{\lam}$ can be obtained by putting $f_{\lam}=\lam$ in \eqref{eq:tilde-f}.
The result is nowadays called Dieudonn\'{e}'s lemma.

In 2013, Rivard \cite{rivard2013application} proved a Dieudonn\'e's lemma of the second order (see also \cite{cho2012multi}).
The original result
can be restated as follows.\\
\noindent{\bf Theorem A}(Rivard \cite{rivard2013application}).
\textit{
Let $\lambda \in \overline{\mathbb{D}}$. Then
\begin{align}
\label{Rivard}
 & \left\{ f''(z_0) : f \in \mathcal{H}_0 (z_0,w_0) \text{ with }
  f'(z_0)
  = \frac{w_0}{z_0} + \frac{|z_0|^2-|w|^2}{\ov{z}_0(1-|z_0|^2)}
  \lambda \right\}
\\
  = \, &
 A(z_0,w_0)
 \overline{\mathbb{D}}(c(\lambda), \rho(\lambda)) ,
\nonumber
\end{align}
where
\[
   A(z_0,w_0)=\frac{2 \left( |z_0|^2-|w_0|^2 \right)}{|z_0|^2(1-|z_0|^2)^2},\quad
   c(\lambda )
  =\lambda \left( 1- \frac{z_0 \ov{w}_0}{\ov{z}_0}\lambda \right) ,\quad
  \rho (\lambda ) = |z_0| (1- |\lambda |^2 ) .
\]
}

For completeness, in \S 2, we shall give an elementary proof of
Theorem A and determine all the extremal functions.

Based on this result, the first author \cite{chen_2019} gave the sharp estimate for $|f''(z_0)|$. In this paper, we would like to further study the second derivative $f''(z_0)$ by explicitly describing the variability region
\begin{equation}\label{variability_region}
V(z_0,w_0)=\{f''(z_0):f\in\H_0(z_0,w_0)\}.
\end{equation}

From  Theorem A it easily follows that
\begin{equation}
  V(z_0,w_0 )
  =
  A(z_0,w_0)
  \bigcup_{\lambda \in \overline{\mathbb{D}}}
  \overline{\mathbb{D}}(c(\lambda), \rho(\lambda)) .
\end{equation}
We note some basic properties of the set $V(z_0,w_0)$.
The class $\mathcal{H}_0(z_0,w_0)$ is a compact convex subset
of the linear space $\mathcal{A}$ of all analytic functions $f$
in $\mathbb{D}$ endowed with the topology of
locally uniformly convergence on $\mathbb{D}$.
The functional $\ell : \mathcal{A} \ni f \mapsto f''(z_0) \in \mathbb{C}$
is continuous and linear.
Therefore the image
$V(z_0,w_0 ) = \ell ( \mathcal{H}_0 (z_0,w_0))$ is also
a compact convex subset of $\mathbb{C}$.
Furthermore the origin is an interior point of $V(z_0,w_0 )$,
because
\[
    A(z_0,w_0)
  \overline{\mathbb{D}}(0, |z_0|)
  =
    A(z_0,w_0)
  \overline{\mathbb{D}}(c(0), \rho(0))
  \subset
  V(z_0,w_0 ).
\]
Recall that a compact convex subset in $\C$
with nonempty interior is a Jordan closed domain (for a proof see \cite[\S 11.2]{berger1987geometry}).
Therefore $\partial V (z_0,w_0)$ is a Jordan curve and
 $V (z_0,w_0)$ is the convex closed domain enclosed
by $\partial V (z_0,w_0)$. \\

Moreover the relations
\begin{equation}
  V(e^{i\theta_1}z_0,e^{i \theta_2}w_0)
  =
  e^{i ( \theta_2 -2 \theta_1 )} V(z_0,w_0 ) ,
  \quad
  \overline{V(z_0,w_0 )} = V( \overline{z_0}, \overline{w_0} )
\end{equation}
hold   for $\theta_1 , \theta_2 \in \mathbb{R}$.
These are  consequences of the facts that
$e^{i \theta_2}f(e^{- i \theta_1}z)
\in \mathcal{H}_0(e^{i \theta_1}z_0,e^{i \theta_2}w_0)$
and that $\overline{f(\overline{z})} \in V(\overline{z_0}, \overline{w_0})$
for $f \in \mathcal{H}_0(z_0,w_0)$, respectively.
In view of these properties it is sufficient to determine
$\partial V(r,s)$ for $0\leq s < r <1$.
In this case we notice that
$\overline{V(r,s)} = V(r,s)$.
Define
\begin{equation}\label{eq:c-s-zeta}
c_s(\zeta)=\zeta(1-s \zeta),\quad \rho_r(\z)=r(1-|\z|^2).
\end{equation}
\begin{thm}
\label{thm:expression-of-boundary-curve}
Let $0 \leq s < r< 1$.
\begin{itemize}
 \item[{\rm (i)}]
If $r+s \leq \frac{1}{2} $,
then $\partial V(r,s)$ coincides with the Jordan curve given by
\begin{equation}\label{eq:i-p-V}
   \partial \mathbb{D} \ni \zeta
    \mapsto A(r,s) c_s(\zeta ).
\end{equation}
 \item[{\rm (ii)}]
If $ r-s \ge \frac{1}{2}$,
then $\partial V(r,s)$ coincides with  the circle given by
\begin{equation}\label{eq:ii-p-V}
   \partial \mathbb{D} \ni \zeta
    \mapsto \frac{1}{2r^2(1-r^2)^2}
    \left[ \left\{ 1+4(r^2-s^2) \right\}r \zeta -s  \right].
\end{equation}
 \item[{\rm (iii)}]
If $r+s > \frac{1}{2}$ and $r-s < \frac{1}{2} $,
then
$\partial V(r,s)$ consists of the circular arc given by
\begin{equation}\label{eq:iii-1-p-V}
   (-\th_0,\th_0)  \ni \th
    \mapsto \frac{1}{2r^2(1-r^2)^2}
    \left[ \left\{ 1+4(r^2-s^2) \right\}r e^{i\th}-s  \right]
\end{equation}
and the simple arc given by
\begin{equation}\label{eq:iii-2-p-V}
    J \ni \zeta
    \mapsto A(r,s) c_s(\zeta ),
\end{equation}
where
\begin{equation}\label{eq:theta-0}
  \theta_0 = \cos^{-1} \frac{r^2+s^2-4(r^2-s^2)^2}{2sr} \in (0, \pi )
\end{equation}
and $J$ is the closed subarc of $\partial \mathbb{D}$ which has
end points
$\zeta_0 = \dfrac{re^{i \theta_0}-s}{2(r^2-s^2)}$
and
$\ov{\zeta}_0 = \dfrac{re^{-i \theta_0}-s}{2(r^2-s^2)}$
and contains $-1$.
\end{itemize}
\end{thm}

We show these three cases of $\p V(r,s)$ in Figure \ref{ab1}, \ref{ab2} and \ref{ab3}.

\begin{figure}[H]
\centering
\subfigure[r=3/4,s=1/4]{
\includegraphics[width=5cm]{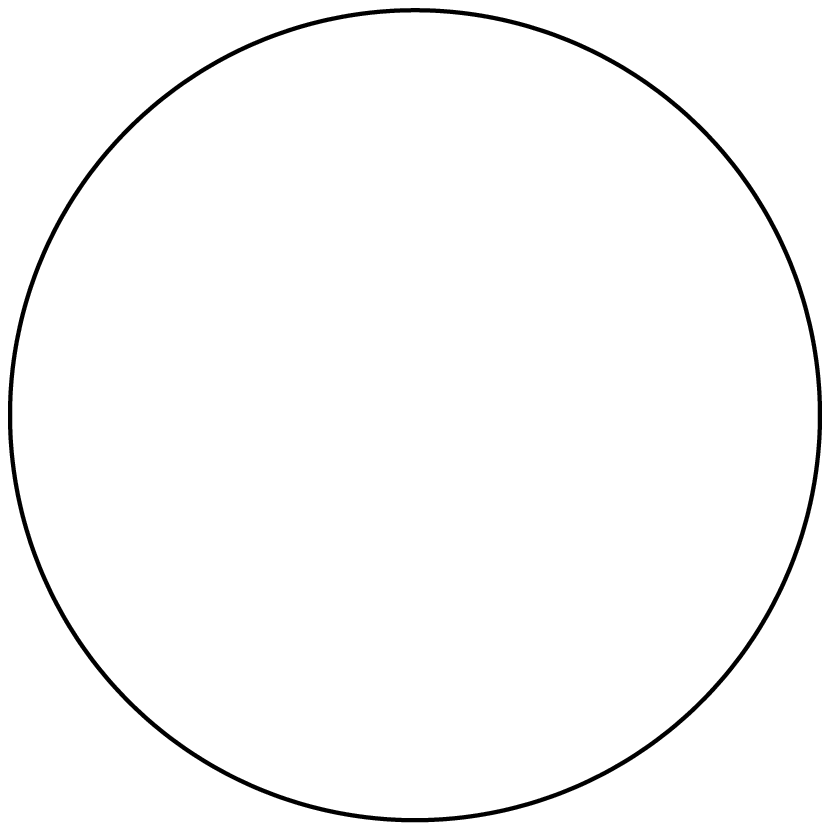}
\label{ab1}
}
\quad
\subfigure[r=1/4, s=4/17]{
\includegraphics[width=5cm]{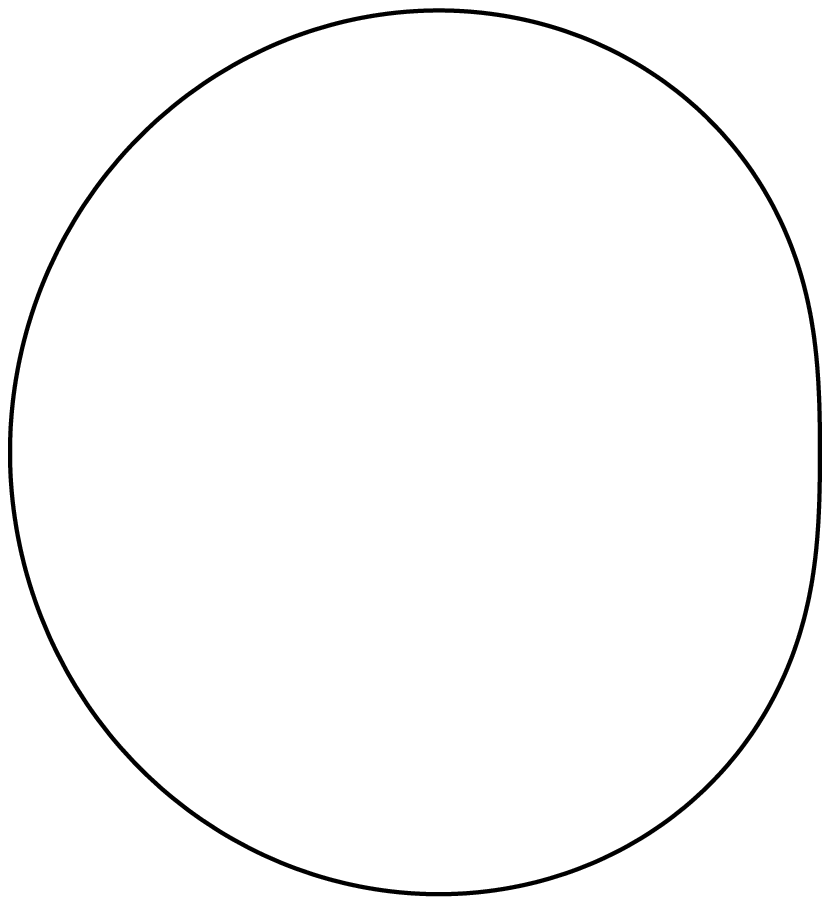}
\label{ab2}
}
\caption{If $r=3/4, s=1/4$, $\p V(r,s)$ is a circle; if $r=1/4, s=4/17$, $\p V(r,s)$ is a convex Jordan curve.}
\end{figure}
\begin{figure}[H]
    \centering 
    \includegraphics[width=7.5cm]{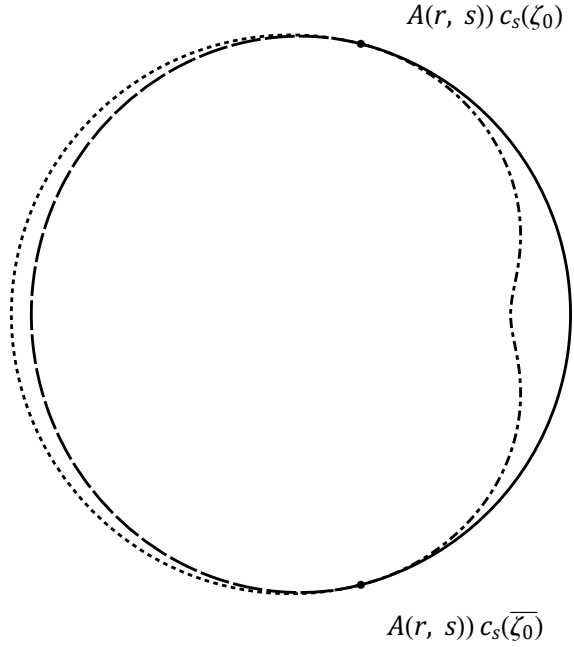}
    \caption{If $r=2/3, s=1/3$, $\p V(r,s)$ consists of a circular arc (solid) and a simple arc (dashed).}
    \label{ab3}
\end{figure}

In fact,
Theorem \ref{thm:expression-of-boundary-curve} is a direct consequence of the following theorem which gives the unified parametric representation of $\p V(r,s)$ and the all extremal functions.
\begin{thm}
\label{thm:boundary-curve}
Let $0 \leq s < r< 1$.
For $\theta \in \mathbb{R}$ let
$r_\theta$ be the unique solution to the
equation
\begin{equation}
\label{eq:definig_equation_of_r_theta}
   |xe^{i\theta} - s| = 2(x^2-s^2) , \quad x > s ,
\end{equation}
if $|re^{i\theta} - s| \geq  2(r^2-s^2) $; otherwise let $r_\theta = r$.
Set
\begin{equation}
\label{eq:def_of_zeta_theta}
  \zeta_\theta = \frac{r_\theta \e -s }{2(r_\theta^2-s^2)}
   \in \overline{\mathbb{D}} .
\end{equation}
 Then a parametric representation
$(- \pi , \pi ] \ni \theta \mapsto \gamma ( \theta )$
of the Jordan curve $\partial V(r,s)$
is given by
 $$\gamma( \theta ) =
     A(r,s)
   \left( c_s( \zeta_\theta )  + \rho_r(\zeta_\theta) e^{i \theta} \right)
  \in \partial V(r,s).$$
Furthermore,
the equality
\[
    f''(r) =  A(r,s)
   \left( c_s( \zeta_\theta )  + \rho_r(\zeta_\theta) e^{i \theta} \right)\in \p V(r,s),
\]
holds for some $\theta \in \mathbb{R}$ with $\zeta_\theta \in \mathbb{D}$
if and only if
\begin{equation}
  f(z)
 =  z T_{\frac{s}{r}}
 \left( T_{-r}(z)  T_{\zeta_\theta } (e^{i \theta} T_{-\zeta_\theta }(z))
 \right) , \quad z \in \mathbb{D} .
\end{equation}
 Here $T_a$ is defined by \eqref{eq:T-a-z}.
Similarly the equality
\[
    f''(r) =  A(r,s)
    c_s( \zeta_\theta )\in \p V(r,s),
\]
holds for some $\theta \in \mathbb{R}$ with $\zeta_\theta \in \partial \mathbb{D}$
if and only if
\begin{equation}
  f(z)
 =  z T_{\frac{s}{r}}
 \left( \zeta_\theta T_{-r }(z)
 \right) , \quad z \in \mathbb{D} .
\end{equation}
\end{thm}

\section{Envelope of a family of circles}
We start this section with the proof of the second order Dieudonn\'e's lemma, which is needed to determine the extremal functions in Theorem \ref{thm:boundary-curve}.
\begin{proof}[Proof of Theorem A]
Let $f\in \H_0(z_0,w_0)$ and define $\tilde{f}$ by \eqref{eq:tilde-f}.
By differentiating both sides of  (\ref{eq:diff-of-g-and-f}) and substituting $z =z_0$
we have
\begin{align}
\label{eq:2nd-diff-of-g-and-f}
    & \frac{2\overline{z_0}\tilde{f}(z_0)}{(1-|z_0|^2)^2}
   + \frac{2\tilde{f}'(z_0)}{1-|z_0|^2} \\
= \, &
 \frac{2 \overline{\left( \frac{w_0}{z_0}\right)}}
{\left( 1 - \left| \frac{w_0}{z_0}\right|^2 \right)^2}
   \left( \frac{z_0f'(z_0)-w_0}{z_0^2} \right)^2
\nonumber \\
& + \frac{1}{1 - \left| \frac{w_0}{z_0}\right|^2}
  \frac{z_0^2f''(z_0)-2(z_0 f'(z_0) - f(z_0)) }{z_0^3}.
\nonumber
\end{align}
Combining this and \eqref{eq:tilde-f-f-z0},
we have
\[
   f''(z_0)
   =
  \frac{2 \left( 1- \left| \frac{w_0}{z_0} \right|^2 \right) }{(1-|z_0|^2)^2}
   \tilde{f}(z_0) \left( 1- z_0 \overline{\left( \frac{w_0}{z_0} \right) } \tilde{f}(z_0) \right)
 + \frac{2 \left( 1- \left| \frac{w_0}{z_0} \right|^2 \right)}{1-|z_0|^2}z_0\tilde{f}'(z_0).
\]
By (\ref{eq:diff-of-g-and-f-at-z_0}),
$  f'(z_0)  = \dfrac{w_0}{z_0}
+ \dfrac{|z_0|^2-|w_0|^2}{\ov{z}_0(1-|z_0|^2)}  \lambda$
holds if and only if $\tilde{f}(z_0) = \lambda$. The Schwarz-Pick inequality
$|\tilde{f}'(z_0)| \leq \dfrac{1-|\tilde{f}(z_0)|^2}{1-|z_0|^2}= \dfrac{1-|\lambda|^2}{1-|z_0|^2}$
implies
$f''(z_0) \in A(z_0,w_0)
 \overline{\mathbb{D}}(c(\lambda), \rho(\lambda))$.

Conversely for $\lambda \in \mathbb{D}$ and $\a\in \ov{\D}$
define analytic functions $\tilde{f}_{\lambda, \a}$ and $f_{\lambda, \a}$
in $\mathbb{D}$ by
\[
  \tilde{f}_{\lambda, \a}(z)
 = T_{\lambda}\left(\frac{|z_0|}{z_0}
 \a T_{-z_0}(z)\right),
 \quad
 f_{\lambda, \a}(z)
 = z T_{\frac{w_0}{z_0}} \left(T_{-z_0}(z) \tilde{f}_{\lambda, \a}(z) \right).
\]
Then $f_{\lambda, \a} \in \mathcal{H}_0(z_0,w_0)$,
$\tilde{f}_{\lambda, \a}(z_0) = \lambda $ and
$f_{\lambda, \a} ''(z_0) = A(z_0,w_0)
\{c(\lambda) + \rho(\lambda) \a\}$. It follows that
$ A(z_0,w_0) \overline{\mathbb{D}} (c(\lambda),  \rho(\lambda)) $ is contained in the variability region.
Furthermore by the uniqueness part of the Schwarz lemma
$f''(z_0) = A(z_0,w_0)
\{c(\lambda) + \rho(\lambda) e^{i \theta }\}$ for some
$f \in \mathcal{H}_0(z_0,w_0)$ if and only if
$f=f_{\lambda,e^{i \th}}$.

Similarly for $\lambda \in \partial \mathbb{D}$ define
$f_\lambda$ by
\[
 f_\lambda(z)
  =z  T_{\frac{w_0}{z_0}} \left( \lambda T_{-z_0}(z)\right).
\]
Then $f_\lambda \in \mathcal{H}_0(z_0,w_0)$ and
$f_\lambda ''(z_0) = A(z_0,w_0) c(\lambda )$.
Again by the uniqueness part of the Schwarz lemma
$f''(z_0) = A(z_0,w_0) c(\lambda )$
for some $f \in \mathcal{H}_0(z_0,w_0)$ if and only if
$f= f_\lambda$.
Thus the proof is completed.
\end{proof}
Let $0 \leq s < r< 1$ and
\begin{equation}
\label{def:tilde-V-r-s}
  \tilde{V}(r,s)
 = \bigcup_{\zeta \in \overline{\mathbb{D}}}
  \overline{\mathbb{D}}(c_s( \zeta), \rho_r (\zeta )) .
\end{equation}
Then by Theorem A,
$V(r,s ) =  A(r,s) \tilde{V}(r,s)$.
Thus the set $\tilde{V}(r,s)$
is a compact and convex subset of $\mathbb{C}$
with $\mathbb{D}(0,r) \subset \tilde{V}(r,s)$.
Therefore $\tilde{V}(r,s)$ is a convex closed Jordan domain
enclosed by the Jordan curve $\partial \tilde{V}(r,s)$.
The determination of $\partial V(r,s)$ is reduced to
that of $\tilde{V}(s,r)$.
\begin{prop}
\label{main-thm}
For $\theta \in \mathbb{R}$ there exists a unique $v_\theta \in \tilde{V}(r,s)$
satisfying
\begin{equation}
\label{ineq:max_of_real_part}
   \Re \{ v e^{-i \theta }\} \leq \Re \{ v_\theta e^{-i \theta }\}
    \quad \text{for all } v \in \tilde{V}(r,s) .
\end{equation}
Furthermore, $v_\theta$ can be expressed as
\begin{equation}
\label{eq:expression_of_v_theta}
   v_\theta
    =
    \begin{cases}
     c_s( \zeta_\theta )  + \rho_r( \zeta_\theta ) e^{i \theta },  \quad &
     |re^{i \theta}-s| < 2(r^2-s^2) , \\
    c_s( \zeta_\theta ) ,  \quad &
     |re^{i \theta}-s| \geq 2(r^2-s^2) ,
    \end{cases}
\end{equation}
where $\zeta_\theta$ is defined in Theorem \ref{thm:boundary-curve}.
\end{prop}
Before proving Proposition \ref{main-thm} we show the following lemma.
\begin{lem}
\label{lemma:monotonicity}
For $\theta \in \mathbb{R}$ and $s\ge 0$, define
a positive and continuous function $h_\theta$ by
\[
   h_\theta (x) = \frac{|xe^{i \theta}-s|}{2(x^2-s^2)}, \quad x > s .
\]
Then $h_\theta $ is strictly decreasing in $x>s$ for each fixed $\theta$
and $\lim\limits_{x \rightarrow \infty} h_\theta (x) = 0$.
\end{lem}
\begin{proof}
The lemma easily follows from
\begin{align*}
   \frac{d \log h_\theta }{dx} (x)
 = \, &
  \frac{2(x-s\cos \theta )}{x^2-2sx \cos \theta +s^2}
 - \frac{2x}{x^2- s^2}
\\
  = \, &
  \frac{-x^3 + 3sx^2 \cos \theta - 3s^2x + s^3 \cos \theta  }
 {(x^2-2sx \cos \theta +s^2)(x^2-s^2)}
\\
  \leq \, &
  \frac{-(x-s)^3}{(x^2-2sx \cos \theta +s^2)(x^2-s^2)} < 0.
\end{align*}
\end{proof}

\begin{proof}[Proof of Proposition \ref{main-thm}]
Let $\theta \in \mathbb{R}$.
We show the existence and uniqueness of $v_\theta \in \partial \tilde{V}(r,s)$
satisfying (\ref{ineq:max_of_real_part})
and that $v_\theta$ can be expressed as (\ref{eq:expression_of_v_theta}).

Since  $\tilde{V}(r,s)$ is compact
and the function $\Re \{ v e^{-i \theta} \}$, $v \in \mathbb{C}$
is nonconstant and harmonic,
there exists $v_\theta \in \partial \tilde{V}(r,s)$
such that
\[
   \Re \{ ve^{-i \theta} \} \leq  \Re \{ v_\theta e^{-i \theta} \}
 \quad \text{for all } v \in \tilde{V}(r,s) .
\]
By (\ref{def:tilde-V-r-s})
there exists $\zeta_\theta^* \in \overline{\mathbb{D}}$
and $\varepsilon_\theta^* \in \partial \overline{\mathbb{D}}$
such that
$v_\theta
= c_s( \zeta_\theta^*) + \rho_\theta (\zeta_\theta^*)\varepsilon_\theta^* $.
By (\ref{def:tilde-V-r-s})
the above inequality is equivalent to
\[
  \Re \{ c_s(\zeta ) e^{-i \theta} \}
 + \rho_r(\zeta)  \Re \{ \varepsilon e^{-i \theta}\}
 \leq
  \Re \{ c_s(\zeta_\theta^* ) e^{-i \theta} \}
 + \rho_r(\zeta_\theta^* )  \Re \{ \varepsilon_\theta^* e^{-i \theta}\} ,
 \; \zeta , \varepsilon \in  \overline{\mathbb{D}} .
\]
By substituting $\zeta = \zeta_\theta^*$ we have
\[
  \rho_r(\zeta_\theta^* )  \Re \{ \varepsilon e^{-i \theta}\}
 \leq
 \rho_r(\zeta_\theta^* )  \Re \{ \varepsilon_\theta^* e^{-i \theta}\}
\]
for all $\varepsilon \in  \overline{\mathbb{D}}$.

If $\zeta_\theta^* \in \mathbb{D}$,
then $ \rho_r(\zeta_\theta^* ) > 0$ and we conclude
$\varepsilon_\theta^* = e^{i \theta }$.
Even if $\zeta_\theta^* \in \partial \mathbb{D}$, since $ \rho_r(\zeta_\theta^* ) =0$,
we may assume $\varepsilon_\theta^* = e^{i \theta }$
and
\begin{equation}
\label{eq:pre_expression}
  v_\theta = c_s(\zeta_\theta^* ) + \rho_r(\zeta_\theta^*) e^{i \theta}.
\end{equation}
Define a continuous function $k_\theta$ on $\overline{\mathbb{D}}$
by
\[
  k_\theta (\zeta ) =   \Re \{ c_s(\zeta ) e^{-i \theta} \}
 + \rho_r(\zeta) .
\]
Then  we have
\begin{equation}
  k_\theta (\zeta )
   \leq
  k_\theta (\zeta_\theta^*  )
 \text{ for all } \zeta  \in  \overline{\mathbb{D}},
\end{equation}
i.e., $k_\theta $ attains a maximum at $\zeta=\zeta_\theta^*$.

Assuming $|re^{i \theta } -s| < 2(r^2-s^2)$
we shall show that $\zeta_\theta^* \in \mathbb{D}$ and
$\zeta_\theta^* = \zeta_\theta$, where $\zeta_\theta$ is defined
in Theorem \ref{thm:boundary-curve}.
Suppose, on the contrary, $\zeta_\theta^* \in \partial \mathbb{D}$.
Putting $\zeta = te^{i \varphi }$
we have
\begin{align}
\label{eq:max_on_the_boundaryI}
  \begin{cases}
   \frac{\partial k_\theta }{\partial \varphi} ( \zeta_\theta^*)
 =  - \Im \{ \zeta_\theta^* (1-2s \zeta_\theta^*) e^{-i \theta }\} = 0 ,
\\
  \frac{\partial k_\theta }{\partial t} ( \zeta_\theta^*)
 = \Re \{ \zeta_\theta^* (1-2s \zeta_\theta^*) e^{-i \theta }\} -2 r \geq 0 .
  \end{cases}
\end{align}
Therefore there exists $r' \geq r$ such that
\begin{equation}
\label{eq:r_primve}
 \zeta_\theta^* (1-2s \zeta_\theta^*) e^{-i \theta } = 2 r' .
\end{equation}
By $\zeta_\theta^* \in \partial \mathbb{D}$
we have
\begin{align}
\label{eqs:r_primve}
\begin{cases}
(1-2s \zeta_\theta^*) e^{-i \theta } = 2 r' \overline{\zeta_\theta^*}, \\
 (1-2s \overline{\zeta_\theta^*}) e^{i \theta } = 2 r' \zeta_\theta^* ,
\end{cases}
\end{align}
which implies
\begin{equation}
\label{eq:zeta_star_and_r_prime}
   \zeta_\theta^* = \frac{r'e^{i \theta}- s}{2(r'^2 - s^2)} .
\end{equation}
By Lemma \ref{lemma:monotonicity}
this implies
\[
  1 = | \zeta_\theta^* |
 = \frac{|r'e^{i \theta}- s|}{2(r'^2 - s^2)}
 \leq \frac{|re^{i \theta}- s|}{2(r^2 - s^2)} < 1,
\]
which is a contradiction. Thus we conclude $\zth^*\in \D$.

Since $k_\theta $ attains a maximum at the interior point
$\zeta_\theta^* \in \mathbb{D}$, we obtain
\begin{equation}\label{eq:p-k-theta}
   \frac{\partial k_\theta}{\partial \zeta } (\zeta_\theta^*)=
\frac{1}{2} (1-2s \zeta_\theta^* ) e^{-i \theta }  - r \overline{\zeta_\theta^*}
 = 0 ,
 \end{equation}
 and hence
\begin{equation}\label{eq:zeta-theta}
 \zeta_\theta^* =
 \frac{re^{i \theta}-s}{2(r^2-s^2)} = \zeta_\theta .
\end{equation}

Similarly assuming $|re^{i \theta } -s| \geq  2(r^2-s^2)$
we shall show that $\zeta_\theta^* \in \partial \mathbb{D}$ and
$\zeta_\theta^* = \zeta_\theta$.
Suppose, on the contrary, $\zeta_\theta^* \in \mathbb{D}$.
Then \eqref{eq:p-k-theta} holds
and hence
\[
     \zeta_\theta^*
  = \frac{re^{i\theta }-s}{2(r^2-s^2)},
\]
which contradicts $| \zeta_\theta^* | < 1 $.
Therefore $\zeta_\theta^* \in \partial \mathbb{D}$.
This implies
(\ref{eq:max_on_the_boundaryI})
and there exists $r' \geq r$ satisfying  (\ref{eq:r_primve})
and (\ref{eq:zeta_star_and_r_prime}).
Again from $\zeta_\theta^* \in \partial \mathbb{D} $
and Lemma \ref{lemma:monotonicity}
it follows that $r'=r_\theta$ and
$\zeta_\theta^* = \zeta_\theta$.

Now we have shown that $\zeta_\theta^* = \zeta_\theta$.
Combining this and (\ref{eq:pre_expression})
$v_\theta$ can be expressed as (\ref{eq:expression_of_v_theta}).
This also implies the uniqueness of $v_\theta$.
\end{proof}
\begin{prop}\label{prop:representation-tilde-V}
The mapping
\[
   (-\pi , \pi ] \ni \th \mapsto v_\theta \in \partial \tilde{V}(r,s)
\]
is a continuous bijection and gives a parametric representation of
$\partial \tilde{V}(r,s)$.
\end{prop}
\begin{rmk}
For a compact convex set $W\subset \C$, a point $w_0\in \p W$ is called a corner point of $W$ if there exists two closed half planes $H_1$ and $H_2$ such that $W\subset H_1 \cap H_2$ and $w_0\in \p H_1 \cap \p H_2$. For details, see \cite[Section 3.4]{pommerenke2013boundary}. In the following proof we show there is no corner points of $\tilde{V}(r,s)$, which is equivalent to the injectivity of the mapping $\th \mapsto v_\theta$.
\end{rmk}
\begin{proof}[Proof of Proposition \ref{prop:representation-tilde-V}]
Firstly, we show the mapping $(-\pi , \pi ] \ni \theta \mapsto v_\theta$
is continuous.
By (\ref{eq:def_of_zeta_theta}) and $r_{-\theta} = r_\theta$
it suffices to show the mapping
$(-\pi , \pi ] \ni \theta \mapsto r_\theta$ is continuous
at any $\theta_0 \in [0, \pi ]$.

(I).  Assume $|re^{i \theta_0} -s| < 2(r^2-s^2)$.
Then $|re^{i \theta} -s| < 2(r^2-s^2)$ holds
on some neighborhood $I$ of $\theta_0$
and hence $r_\theta \equiv r$ on $I$. Thus
$r_\theta$ is continuous at $\theta_0$.

(II). Assume $|re^{i \theta_0} -s| > 2(r^2-s^2)$.
Then $|re^{i \theta} -s| > 2(r^2-s^2)$ holds
on some neighborhood $I$ of $\theta_0$.
Hence $r_\theta $ is the unique solution to
the equation (\ref{eq:definig_equation_of_r_theta}),
which is equivalent to
$h_\theta (x) -1 = 0$.
In this case the continuity of $r_\theta$ at $\theta_0$
is a consequence of the inequality $\dfrac{dh_\theta}{dx} (x) <0$ (see Lemma \ref{lemma:monotonicity}) and the implicit function theorem.

(III).  Assume $|re^{i \theta_0} -s| = 2(r^2-s^2)$.
As in the case (II) there exists a neighborhood $I$ of $\theta_0$
the equation $h_\theta (x)-1=0$ has
the unique solution $x(\theta)$ which is continuous in $\theta$
and $x(\theta_0) = r$.
Since $|re^{i \theta} -s| < 2(r^2-s^2)$ for $\theta \in I_1 := I \cap [0, \theta_0 )$,
we have $r_\theta \equiv r$ on $I_1$.
Similarly  since $|re^{i \theta} -s| > 2(r^2-s^2)$
for $\theta \in I_2 := I \cap (\theta_0 , \pi ] $,
we have $r_\theta \equiv x( \theta )$ on $I_2$.
Therefore $r_\theta$ is continuous at $\theta = \theta_0$, as required.

Secondly, we show the mapping $(-\pi , \pi ] \ni \theta \mapsto v_\theta$
is injective.
Suppose,  on the contrary, $v_{\theta_1}  = v_{\theta_2} = v^*$
for some $- \pi < \theta_1 < \theta_2 \leq \pi$.
For $j=1,2$ define a half plane $H_j$ by
$$H_j =\{ w \in \mathbb{C}:
  \Re ( we^{-i \theta_j } ) \leq \Re ( v^* e^{-i \theta_j } )\}.
$$
Then $v^* \in \p\tilde{V}(r,s) \subset H_1 \cap H_2$,
$v^* \in \partial H_1 \cap \partial H_2$.
And the opening angle $\alpha$ of $H_1\cap H_2$ is less than $\pi$, which means $v^*$ is a corner point of $\tilde{V}(r,s)$.
Take  $\z^*\in \ov{\D}$ and $\th^*\in \R$ such that
$v^*=c_s(\z^*)+\rho_r(\z^*)e^{i\th^*}$.

 If $\z^*\in \D$, then the circle $ \p \D(c_s(\z^*),\rho_r(\z^*) )$ is contained in $\tilde{V}(r,s)$ and passes through $v^*\in \p \tilde{V}(r,s)$. This contradicts $\alpha<\pi$.

 Assume $\z^*\in \p \D$. Then the curve $\{c_s(\z):\z\in \p \D\}$ passes through $v^*=c_s(\z^*)$ and is contained in $\tilde{V}(r,s)$.
 If $c_s'(\z^*)\neq 0$, then we have a contradiction as before.
 Notice that $c_s'(\z^*)= 0$ if and only if $s=\frac{1}{2}$ and $\z^*=1$.
  In this case, since $r>s=\frac{1}{2}$, $v^*=c_{\frac{1}{2}}(1)=\frac{1}{2}\in \D(0,r)\subset \interior \tilde{V}(r,s)$, which is also a contradiction.
  \begin{figure}
  \centering
\begin{tikzpicture}
\draw (0,0) to [out=-45, in=300] (6,2);
\draw (0,0) to [out=135, in=150] (6,2);
\draw (3,{2+3*sqrt(3)/3})--(6,2)--(8,{2*sqrt(3)/3}) ;
\node[above] at (3,{2+3*sqrt(3)/3}) {$\partial H_2$};
\draw (4,{2+2*sqrt(3)})--(6,2)--(7.5,{2-1.5*sqrt(3)});
\node at (5,{2+1.5*sqrt(3)}) {$\partial H_1$} ;
\node at (2,1) {$\tilde{V}(r,s)$};
\draw[fill] (6,2) circle (1pt);
\node[right] at (6,2.2) {$v^*$};
\draw[thick,dashed] (6,2) -- (6,5);
\draw[->,domain=90:120] plot ({6+0.5*cos(\x)}, {2+0.5*sin(\x)});
\node at (5.8,2.7){$\theta_1$};
\draw[->,domain=90:150] plot ({6+1*cos(\x)}, {2+1*sin(\x)});
\node at (5.1,3) {$\theta_2$};
\draw[->,thick,domain=150:300] plot ({6+0.75*cos(\x)}, {2+0.75*sin(\x)});
\node at (5,1.5) {$\alpha$};
\end{tikzpicture}
\end{figure}

We have shown that the mapping
$(-\pi , \pi ] \ni \theta \mapsto v_\theta \in \partial \tilde{V}(r,s)$
is continuous and injective.
Since $\partial \tilde{V}(r,s)$ is a Jordan curve,
by making use of the intermediate value theorem, one can easily conclude
the mapping is also surjective.
Therefore the mapping gives a parametric representation
of $\partial \tilde{V}(r,s)$.
\end{proof}
\begin{rmk}
Notice that $\Gamma_s=\{c_s(\z):\z\in \p \D\}$ is convex if $0\le s\le 1/4$, is smooth and non-convex if $1/4<s<1/2$, has a cusp if $s=1/2$, and has a self-intersection point if $1/2<s<1$.
\end{rmk}

\section{Proof of Theorems \ref{thm:expression-of-boundary-curve} and \ref{thm:boundary-curve}}
 In this section we begin with the proof of  Theorem \ref{thm:boundary-curve}, which directly leads to  Theorem \ref{thm:expression-of-boundary-curve}.
\begin{proof}[Proof of Theorem \ref{thm:boundary-curve}]
Recall that $V(r,s ) =  A(r,s) \tilde{V}(r,s)$. Then by Proposition \ref{prop:representation-tilde-V} we conclude that the mapping
$$\gamma(\th)=A(r,s) \vth=A(r,s) (c_s( \zeta_\theta )  + \rho_r( \zeta_\theta ) e^{i \theta })$$
 gives the parametric representation of $\p V(r,s)$.

 From the argument of the proof of Theorem A, we can easily get the all extremal functions as required.
\end{proof}

\begin{proof}[Proof of Theorem \ref{thm:expression-of-boundary-curve}]
Noting
 $$\frac{1}{2(r+s)} = \frac{r-s}{2(r^2-s^2)}\le\frac{|r e^{i \theta}-s|}{2(r^2-s^2)}\le \frac{r+s}{2(r^2-s^2)}=\frac{1}{2(r-s)},$$
 we consider the following three cases.

(\rm{i})
If $s+r\le \frac{1}{2}$, then $\dfrac{|r e^{i \theta}-s|}{2(r^2-s^2)}\ge 1$ always hold for $\th \in(-\pi,\pi]$. Thus $\zth \in \p\D$, $\vth=c_s( \zeta_\theta )$ and
  $\p V(r,s)$ is given by
$$A(r,s)
  c_s( \zeta_\theta ),\quad \th \in \R.$$
As a function of $\th$, $\zth$ is continuous on $(-\pi,\pi]$ and injective since $\vth$ is injective. Thus, the mapping $\p \D\ni e^{i \th}\mapsto \zth \in \p \D$ is surjective and hence homeomorphic. Therefore, the map $\p \D\ni\z\mapsto A(r,s)
  c_s( \zeta)$
  is an another parametric representation of
$\p V(r,s)$.

(\rm{ii})
    If $r-s\ge \frac{1}{2}$, then $\dfrac{|r e^{i \theta}-s|}{2(r^2-s^2)}\le1$ always hold for $\th \in(-\pi,\pi]$. Thus  $\z_\theta = \dfrac{r e^{i \theta}-s}{2(r^2-s^2)}$, $\vth=c_s( \zeta_\theta )  + \rho_r( \zeta_\theta ) e^{i \theta }$. And
     $\p V(r,s)$ is given by
    $$A(r,s) \vth, \quad \th\in(-\pi,\pi].$$
Since
\begin{align*}
\vth&=\frac{r e^{i \theta}-s}{2(r^2-s^2)}\left(1-\frac{r e^{i \theta}-s}{2(r^2-s^2)} s\right)+r\left(1-\frac{r e^{i \theta}-s}{2(r^2-s^2)}\frac{r e^{-i \theta}-s}{2(r^2-s^2)}\right)\e\\
&=r\e+\frac{r e^{i \theta}-s}{2(r^2-s^2)}-\frac{(r e^{i \theta}-s)(s(r e^{i \theta}-s)+r(r-s\e))}{4(r^2-s^2)^2}\\
&=r\e+\frac{r e^{i \theta}-s}{2(r^2-s^2)}-\frac{(r e^{i \theta}-s)(r^2-s^2)}{4(r^2-s^2)^2}\\
&=\frac{\left\{ 1+4(r^2-s^2) \right\}r \e -s}{4(r^2-s^2)},
\end{align*}
we conclude that $\p V(r,s)$  coincides with  the circle given by \eqref{eq:ii-p-V}.

(\rm{iii})
If $r-s< \frac{1}{2}$ and $s+r> \frac{1}{2}$, then
$|r e^{i \theta}-s|=2(r^2-s^2)$ has the unique solution $\th_0=\cos^{-1} \frac{r^2+s^2-4(r^2-s^2)^2}{2sr}\in(0,\pi)$.
For $|\th|<\th_0$, we have $\dfrac{|r e^{i \theta}-s|}{2(r^2-s^2)} < 1$. Thus $\z_\theta = \dfrac{r e^{i \theta}-s}{2(r^2-s^2)}\in \D$ and $\vth=c_s( \zeta_\theta )  + \rho_r( \zeta_\theta ) e^{i \theta }$.
For $\th_0\le|\th|\le \pi$, we have $\dfrac{|r e^{i \theta}-s|}{2(r^2-s^2)} \geq 1$. Thus $\zth\in \p \D$ and $\vth=c_s( \zeta_\theta )$.
Therefore,  $\p V(r,s)$ consists of the following two curves:

(a) If $|\th|<|\th_0|$, then
\begin{align*}
\gamma(\th)&=\frac{2 \left( r^2-s^2 \right)}{r^2(1-r^2)^2}
   \left( c_s( \zeta_\theta )  + \rho_r(\zeta_\theta) e^{i \theta} \right)\\
   &=\frac{1}{2r^2(1-r^2)^2}
    \left[ \left\{ 1+4(r^2-s^2) \right\}r e^{i \theta } -s  \right],
\end{align*}
coincides with \eqref{eq:iii-1-p-V}.

(b)
If $|\th|\ge|\th_0|$, then
$\gamma(\th)=\dfrac{2 \left( r^2-s^2 \right)}{r^2(1-r^2)^2}
  c_s( \zeta_\theta )$,
  where $\zth$ is defined as in Theorem \ref{thm:expression-of-boundary-curve}.
Notice that the map $\zth$ is continuous and injective with respect to $\th\in (-\pi,\pi]$ and $\z_{\pi}=-1$. Therefore, the set $\{\zth:|\th_0|\le|\th|\le \pi\}$ coincides with the closed subarc $J$ of $\p \D$ which has
end points
$\zeta_0 = \dfrac{re^{i \theta_0}-s}{2(r^2-s^2)}$
and
$\ov{\zeta}_0 = \dfrac{re^{-i \theta_0}-s}{2(r^2-s^2)}$
and contains $-1$.
\end{proof}
\subsection*{Acknowledgements}
The authors would like to thank Professor Toshiyuki Sugawa for his proposal for this topic, his helpful comments and valuable suggestions. The authors are also grateful to Prof. Ikkei Hotta and Prof. Masahiro Yanagishita for their helpful suggestions.
\bibliographystyle{amsplain} 

\end{document}